\documentclass[11pt]{amsart}
\usepackage{amsfonts,amssymb,amscd,amstext,mathrsfs,mathtools}
\usepackage[utf8]{inputenc}
\usepackage{hyperref}
\usepackage{verbatim}

\usepackage{graphics}
\usepackage{graphicx}

\usepackage{times}
\usepackage{enumerate}
\usepackage[up,bf]{caption}
\usepackage{color}
\usepackage{t1enc}
\usepackage{array}

\input xy
\xyoption{all}

\usepackage[color=blue!20]{todonotes}

\numberwithin{equation}{section}

\newtheorem{theorem}{Theorem}[section] 
\newtheorem{proposition}[theorem]{Proposition} 
 
\newtheorem{lemma}[theorem]{Lemma} 
\newtheorem{corollary}[theorem]{Corollary} 

\theoremstyle{definition} 
 
\newtheorem{remark}[theorem]{Remark} 
 
\newtheorem{problem}[theorem]{Problem}

\newcommand\bA{\mathbf{A}}

\newcommand\Ecal{\mathcal{E}}

\newcommand\Pcal{\mathcal{P}}

\newcommand\Ascr{\mathscr{A}} 
 
\newcommand\Cscr{\mathscr{C}}

\newcommand\Gscr{\mathscr{G}}

\newcommand\Oscr{\mathscr{O}}

\newcommand\C{\mathbb{C}} 
 
\newcommand\CP{\mathbb{CP}}

\newcommand\N{\mathbb{N}} 
 
\newcommand\R{\mathbb{R}}

\newcommand\Z{\mathbb{Z}} 

\newcommand\igot{\mathfrak{i}}

\renewcommand\igot{\mathfrak{i}}

\newcommand\qgot{\mathfrak{q}}

\renewcommand\imath{\igot}

\newcommand\hra{\hookrightarrow}

\newcommand\longhookrightarrow{\ensuremath{\lhook\joinrel\relbar\joinrel\rightarrow}} 

\newcommand\wt{\widetilde} 
 
\newcommand\di{\partial} 
\newcommand\dibar{\overline\partial}

\newcommand\CMI{\mathrm{CMI}} 
 
\newcommand\TC{\mathrm{TC}}

\newcommand\supp{\mathrm{supp}}

\def\br{\mathrm{br}} 
\def\s{\mathrm{s}} 
\def\f{\mathrm{f}} 
\def\rank{\mathrm{rank}} 
\def\Ell1{\mathrm{Ell_1}} 
\def\CEll1{\mathrm{CEll_1}}

\makeatletter
\@namedef{subjclassname@2020}{\textup{2020} Mathematics Subject Classification}
\makeatother

\begin{document} 

\title{Removing singularities of minimal surfaces by isotopies} 

\author{Antonio Alarc\'on and Franc Forstneri\v c}

\address{Franc Forstneri\v c, Faculty of Mathematics and Physics, University of Ljubljana, Jadranska 19, SI--1000 Ljubljana, Slovenia}

\address{Franc Forstneri\v c, Institute of Mathematics, Physics and Mechanics, Jadranska 19, SI--1000 Ljubljana, Slovenia}

\email{franc.forstneric@fmf.uni-lj.si}

\address{Antonio Alarc\'on, Departamento de Geometr\'{\i}a y Topolog\'{\i}a e Instituto de Matem\'aticas (IMAG), Universidad de Granada, Campus de Fuentenueva s/n, E--18071 Granada, Spain.}

\email{alarcon@ugr.es}


\date{18 November 2025}


\begin{abstract}
Given an open Riemann surface $M$, we show that the branch points 
and the complete ends of finite total curvature of a conformal minimal surface
$M\to\R^n$, $n\ge 3$, can be removed by an isotopy through such surfaces. 
The analogous result holds for null holomorphic curves $M\to\C^n$.
\end{abstract}

\maketitle


\section{Introduction}\label{sec:intro}

\noindent
Let $M$ be a connected open Riemann surface. 
Recall that an immersion $u=(u_1,\ldots,u_n):M\to\R^n$ $(n\ge 3)$ is 
{\em conformal} (angle preserving) if and only if its $(1,0)$-differential 
$\partial u=(\partial u_1,\ldots,\partial u_n)$ (the $\C$-linear part of 
the differential $du=\di u+\dibar u$) satisfies the nullity condition
\begin{equation}\label{eq:nullity}
	(\partial u_1)^2+\cdots+(\partial u_n)^2=0. 
\end{equation}
(See e.g.\ \cite{Osserman1986} or \cite{AlarconForstnericLopez2021}.)
A conformal immersion $u:M\to\R^n$ parametrises a minimal surface 
in $\R^n$ with the Euclidean metric $ds^2$ if and only if it is harmonic, 
if and only if $\di u$ is a holomorphic $1$-form. 
Assuming that this holds and choosing a nowhere 
vanishing holomorphic $1$-form $\theta$ in $M$, we have 
$2\di u=f\theta$ where $f:M\to\C^n$ is a holomorphic map with values
in the punctured null quadric $\bA_*=\bA\setminus \{0\}$, where 
\begin{equation}\label{eq:nullq}
	\bA=\{z=(z_1,\ldots,z_n)\in\C^n : z_1^2+\cdots+z_n^2=0\}.
\end{equation}
Given any point $x_0\in M$, we recover $u$ from its 
Weierstrass data $f\theta$ by
\begin{equation}\label{eq:W-formula}
	u(x)=u(x_0)+ \Re \int_{x_0}^x 2\di u = u(x_0)+ \Re \int_{x_0}^x f\theta,
	\quad x\in M.
\end{equation}
Here, $\Re$ denotes the real part.
Conversely, a holomorphic map $f:M\to  \bA_*$ satisfying 
$\Re \oint_\gamma f\theta=0$ for every closed curve $\gamma\subset M$ 
(that is, $\Re(f\theta)$ is exact on $M$) determines a conformal minimal 
immersion $u:M\to\R^n$ by the above formula. 
The {\em generalised Gauss map}, or simply the {\em Gauss map}, 
of $u$ is the holomorphic map
\begin{equation}\label{eq:Q}
	\Gscr(u):M\to Q=\bigl\{ [z_1:\cdots:z_n]\in \CP^{n-1}: z_1^2+\cdots + z_n^2=0\big\}
\end{equation}
given by 
\begin{equation}\label{eq:GaussMap}
	\Gscr(u)(p)=[\di u_1(p):\cdots:\di u_n(p)], \quad p\in M.
\end{equation}

In this paper, we consider minimal surfaces with 
isolated singularieties. One type of singularities are branch points.
Let $u:M\to\R^n$ be a $\Cscr^1$ map with rank $2$ at some point. 
Denote by $\br(u)\subsetneq M$ the set of points $x\in M$ at which 
$u$ is not an immersion, i.e., $\rank \, du_x<2$. 
If the immersion $u:M\setminus \br(u)\to \R^n$ is 
conformal and harmonic, then $\di u$ is a continuous $(1,0)$-form on $M$ 
which is holomorphic on $M\setminus \br(u)$ and
satisfies $\{\di u=0\}=\br (u)$. 
By a theorem of Rad{\'o} \cite{Rado1924} 
(see also \cite[Theorem 3.4.17]{Stout2007}),
$\di u$ extends holomorphically to $M$, so $\br (u)$ 
is a closed discrete subset of $M$ and $u$ is harmonic on $M$.
The points of $\br (u)$ are called {\em branch points} of $u$, 
and $u$ is said to be a {\em branched conformal minimal surface};
see e.g.\ \cite[Ch.\ 6]{Osserman1986} or 
\cite[Remark 2.3.7]{AlarconForstnericLopez2021}.
Branch points of minimal surfaces are not removable by small deformations 
\cite[Remark 3.12.6]{AlarconForstnericLopez2021}. 
Our first result is that they are removable by isotopies. 
It is proved in Section \ref{sec:branched}; 
see the more precise statement in Theorem \ref{th:branched-2}.

%
%
\begin{theorem}\label{th:branched}
Given a branched conformal minimal surface $u:M\to\R^n$, 
there is an isotopy of branched conformal 
minimal surfaces $u_t:M\to\R^n$, $t\in [0,1]$, such that $u_0=u$ and 
$u_1$ is an immersion everywhere on $M$, that is, $\br(u_1)=\varnothing$. 
Furthermore, we can choose the isotopy such that for each $t\in [0,1]$ the 
Gauss map $\Gscr(u_t)$ \eqref{eq:GaussMap} of $u_t$ 
equals $\Gscr(u)$ in their common domain 
of definition $M\setminus(\br(u_t)\cup\br(u))$.
\end{theorem}

By an {\em isotopy}, we mean a family of maps 
depending continuously on a parameter $t\in[0,1]$. The space 
\[
	\CMI_\br(M,\R^n)
\]
of branched conformal minimal surfaces $M\to\R^n$ is endowed 
with the compact-open topology and contains the subspace
$\CMI(M,\R^n)$ of conformal minimal immersions $M\to\R^n$.
Recall that $u\in \CMI_\br(M,\R^n)$ is said to be {\em nonflat}
if and only if $u(M)$ is not contained in an affine plane of $\R^n$;  
equivalently, the image of the map 
$f=2\di u/\theta:M\to\bA$ is not contained 
in a ray of $\bA$ \eqref{eq:nullq}. Also, $u$ is called {\em full} if and only if
$f(M)$ is not contained in a proper linear subspace of $\C^n$.
(See Definition 2.5.2 and Lemma 2.5.3 in
\cite[p.\ 106]{AlarconForstnericLopez2021}.) 
Note that the second assertion in Theorem \ref{th:branched} implies that if the given surface $u$ is nonflat (resp.\ full) then the isotopy $u_t$ $(t\in [0,1])$ can be chosen to consist of nonflat (resp.\ full) surfaces.

%
%
Another important type of isolated singularities of minimal surfaces are
{\em complete ends of finite total curvature}. The Gaussian curvature
of a smooth immersed surface $u:M\to\R^n$ 
is a function $K:M\to \R$ whose value at $p\in M$
is the Gauss curvature of the Riemannian metric $u^*ds^2$ at $p$.
If $u\in \CMI(M,\R^n)$ is a minimal surface then $K$ assumes 
values in $\R_-=(-\infty,0]$, and the total curvature is the number
$\TC(u)=\int_M K\, dA\in [-\infty,0]$, where $dA$ is the area measure
determined by $u^*ds^2$. 
(See \cite[Sect.\ 2.6]{AlarconForstnericLopez2021}.) 
We say that $u$ is of {\em finite total curvature} if $\TC(u)>-\infty$.
A minimal surface $u:M\to\R^n$ is said to be {\em complete} if the 
metric $u^*ds^2$ induces a complete distance function on $M$.
If $M$ is a bordered Riemann surface with compact closure $\overline M$, 
$P \subset M$ is a compact subset, and 
$u: \overline M\setminus P\to\R^n$ is a complete conformal minimal surface 
of finite total curvature, then $P$ is a finite set 
by a theorem of Huber \cite{Huber1957} (see also \cite[Theorem 2.6.4]{AlarconForstnericLopez2021}), $\di u$ extends to a meromorphic 
1-form on $M$ with a pole of order $\ge 2$ at every point of $P$ by the 
Chern--Osserman theorem \cite{ChernOsserman1967JAM} 
(see also \cite[Theorem 4.1.1]{AlarconForstnericLopez2021}), 
$u$ is proper at every end $p\in P$, and 
its asymptotic behaviour at $p$ is described by the 
Jorge--Meeks theorem \cite{JorgeMeeks1983T} 
(see also \cite[Theorem 4.1.3]{AlarconForstnericLopez2021}).
Conversely, a nontrivial meromorphic $1$-form 
$\phi=(\phi_1,\ldots,\phi_n)$ on an open Riemann surface $M$ 
(such $\phi$ is called an {\em abelian differential}) 
has a closed discrete polar locus $P(\phi)\subset M$. If $\phi$
satisfies the nullity condition \eqref{eq:nullity} and has vanishing real 
periods on closed curves in $M'=M\setminus P(\phi)$, 
then it determines a conformal minimal 
surface $u:M'\to\R^n$ by $u(x)=\Re \int^x \phi$ with a complete 
end of finite total curvature at each point of $P(\phi)$. Let 
\[
	\CMI_\s(M,\R^n)
\] 
denote the space of conformal minimal immersions
$u:M\setminus P \to\R^n$, where $P=P(u)$ is a closed discrete subset 
of $M$ and $\di u$ is meromorphic on $M$ with an 
effective pole at every point $p\in P$ and no other zeros or poles. 
(The subscript s stands for singularities.)
With $\theta$ as above, we have $2\di u=f\theta$ where $f$ is a 
meromorphic map on $M$ with values in $\bA_*$ \eqref{eq:nullq} 
whose polar locus is $P$. (See Subsect.\ \ref{ss:abelian}.)
We can view $f$ as a holomorphic map to the complex submanifold 
\begin{equation}\label{eq:Y}
	Y =\bA_*\cup Q 
\end{equation}
of $\C^n \cup\CP^{n-1}=\CP^n$, where $Q$ is the hyperquadric in \eqref{eq:Q}.
The topology on $\CMI_\s(M,\R^n)$ is defined by the 
condition that a sequence $u_j \in \CMI_\s(M,\R^n)$ converges to 
$u\in \CMI_\s(M,\R^n)$ if and only if $f_j=2\di u_j/\theta:M\to Y$ 
converges to $f=2\di u/\theta:M\to Y$ in the compact-open topology 
and there is a point $x_0 \in M\setminus P(f)$ 
such that $u_j(x_0)$ converges to $u(x_0)$ in $\R^n$.
Note that $\CMI_\s(M,\R^n)$ contains the subspace $\CMI(M,\R^n)$
with its usual compact-open topology.
We denote by 
\[
	\CMI_\br^\f(M,\R^n)\subset \CMI_\br(M,\R^n),\ \   
	\CMI_\s^\f(M,\R^n)\subset \CMI_\s(M,\R^n), \ \ 
        \CMI^\f(M,\R^n)\subset \CMI(M,\R^n)
\]
the corresponding subspaces of full maps.

The order of the pole of $f$ at $p\in P$ 
is the local intersection number of $f$ with $Q$ at $p$, which 
is positive, so a pole cannot be removed
by a small deformation of $u$.
Our second main result is that complete ends of finite total curvature of 
full minimal surfaces can be removed by an isotopy. 
It is proved in Sect.\ \ref{sec:FTC}. 

%
%
\begin{theorem}\label{th:sing}
Let $M$ be an open Riemann surface and $n\ge 3$ an integer. For any 
$u\in \CMI_\s^\f(M,\R^n)$ there is an isotopy $u_t\in \CMI_\s^\f(M,\R^n)$, 
$t\in [0,1]$, such that $u_0=u$ and $u_1$ is defined everywhere on $M$, that is, $u_1\in \CMI^\f(M,\R^n)$.
\end{theorem}

We wish to explain the reason for a somewhat different assumption on the 
initial minimal surface $u$ in Theorem \ref{th:branched} 
(where $u$ need not be full) and in Theorem \ref{th:sing},  
where $u$ is assumed to be full.
The proof of Theorem \ref{th:branched} amounts to finding a path of holomorphic 
abelian differentials $\omega_t$ $(t\in [0,1])$ on $M$ with values in $\bA$ 
\eqref{eq:nullq} and vanishing real periods such that $\omega_0=2\partial u$ 
and $\omega_1$ has no zeros. Integrating these abelian differentials by the
Weierstrass formula \eqref{eq:W-formula} gives an isotopy of conformal minimal 
surfaces $u_t(x)=\Re\int^x\omega_t$ satisfying the conclusion of 
Theorem \ref{th:branched}. The proof is accomplished in two steps. 
In the first step (see Proposition \ref{prop:abelian} (b))
we find a path of nontrivial abelian differentials 
$\varpi_t=h_t\omega_0$ $(t\in [0,1])$ with values in $\bA$, 
where $h_t$ is a path of meromorphic functions on $M$, satisfying 
$\varpi_0=\omega_0$ (that is, $h_0\equiv 1$) but without paying 
attention to the period conditions. This does not require fullness.
In the second step we find a path of nowhere vanishing holomorphic 
functions $\xi_t$ on $M$, with $\xi_0\equiv 1$, such that the path 
$\omega_t=\xi_t\varpi_t$ $(t\in [0,1])$ satisfies that each $\omega_t$ 
has vanishing real periods; see Proposition \ref{pr:multipliers}. 
This does not required fullness either; see Remark \ref{rem:nofull}. 
The second assertion in Theorem \ref{th:branched} is granted in this 
construction since each $\omega_t$ is of the form 
$\omega_t=2\xi_i h_t\partial u$; see \eqref{eq:GaussMap}. 
Only Runge approximation for functions into 
$\C^*=\C\setminus\{0\}$ is required for this task. 
The proof of Theorem \ref{th:sing} follows the same scheme but now 
the path of abelian differentials $\omega_t$ on $M$ is constructed 
to be nowhere vanishing and such that $\omega_1$ is holomorphic on $M$ 
(that is, it has no poles). In this case, the period problem is considerably 
more delicate than in the context of Theorem \ref{th:branched}, and in 
order to solve it we use Runge approximation for maps into the 
Oka manifold $Y$ in \eqref{eq:Y}. This forces us to ask that 
the given surface $u$ in Theorem \ref{th:sing} be full, and prevents us to 
preserve the Gauss map along the isotopy.

Going further, the proofs of Theorems \ref{th:branched} and \ref{th:sing}
show that the two results can be combined, that is, we can push 
both the branch points and the ends of finite total curvature 
out of the surface by an isotopy. Let 
\[
	\CMI_{\br,\s}(M,\R^n)\supset \CMI_{\br,\s}^\f(M,\R^n)
\] 
denote the set of conformal minimal immersions 
$u:M\setminus \Ecal_u\to\R^n$, where 
\begin{equation}\label{eq:Ecal}
	\Ecal_u=\Ecal_u^0\cup \Ecal_u^\infty = \br(u) \cup P(u) 
\end{equation}
is a (possibly empty) closed discrete subset of $M$ (depending on $u$) 
such that $\partial u$ is meromorphic on $M$ with the zero set 
$\Ecal_u^0=\br(u)$ (the branch locus of $u$) 
and the polar locus $\Ecal_u^\infty = P(u)$
(the set of complete ends of finite total curvature of $u$), and its subset of full maps. 
We call $\Ecal_u$ the {\em singular locus} of $u$.
The topology on $\CMI_{\br,\s}(M,\R^n)$
is determined in the same way as on its subspace $\CMI_{\s}(M,\R^n)$.

The following is our third main result.

%
%
\begin{corollary}\label{cor:main}
Let $M$ be an open Riemann surface and $n\ge 3$ an integer.
For every $u\in \CMI_{\br,\s}^\f(M,\R^n)$ there is an isotopy 
$u_t \in \CMI_{\br,\s}^\f(M,\R^n)$, $t\in [0,1]$, such that $u_0=u$ 
and $u_1$ is an immersion defined everywhere on $M$, 
that is, $u_1\in\CMI^\f(M,\R^n)$. 
\end{corollary}

Every flat conformal minimal immersion is isotopic to a nonflat one
\cite[Theorem 5.7.6]{AlarconForstnericLopez2021}, hence the immersion 
$u_1$ in Theorem \ref{th:branched} can be chosen nonflat whenever 
one does not insist on the condition on the Gauss map. Furthermore, 
by a recent result of Vrhovnik \cite{Vrhovnik2025}, every nonflat conformal
minimal immersion $M\to\R^n$, $n\ge 3$, is isotopic to a proper one, 
which can be chosen an immersion with simple double points if $n=4$ 
and an embedding if $n\ge 5$. Therefore, the immersions $u_1$
in Theorems \ref{th:branched} and \ref{th:sing}, and in Corollary \ref{cor:main}, 
can be chosen of this kind. 

An immediate consequence of Theorems \ref{th:branched},
\ref{th:sing}, and Corollary \ref{cor:main} is that each of 
the inclusions in 
\begin{equation}\label{eq:inclusions}
\xymatrix{  
			    &  \CMI_\br^\f(M,\R^n) \ar@{^{(}->}[rd] & \\
	\CMI^\f(M,\R^n) \ar@{^{(}->}[ru] \ar@{^{(}->}[rd] 
			    & & \CMI_{\br,\s}^\f(M,\R^n) \\
			    &  \CMI_\s^\f(M,\R^n) \ar@{^{(}->}[ru] & \\
}
\end{equation}
induces a surjection of path components. The same holds for the inclusion
\begin{equation}\label{eq:inclusions-2}
	\CMI(M,\R^n)\longhookrightarrow \CMI_\br(M,\R^n)
\end{equation}
by Theorem \ref{th:branched}. Recall that for any open Riemann surface $M$ 
we have $H_1(M,\Z)=\Z^l$, $l\in \Z_+\cup\{\infty\}$.   
The set of path components of the space $\CMI^\f(M,\R^3)$ is in 
bijective correspondence with the $2^l$ elements of the group $(\Z_2)^l$
(see \cite[Corollary 1.6]{ForstnericLarusson2019CAG} or  
\cite[Corollary 3.12.4]{AlarconForstnericLopez2021}), and 
$\CMI^\f(M,\R^n)$ is path connected for $n\ge 4$ by 
\cite[Theorem 6.1]{AlarconLarusson2025}. For the space of
nonflat conformal minimal immersions, the same holds by 
\cite[Corollary 1.6]{ForstnericLarusson2019CAG}, and for 
$\CMI(M,\R^n)$ it holds by \cite[Corollary 5.7.7]{AlarconForstnericLopez2021}.
This implies the following corollary.

\begin{corollary}\label{cor:path-components}
Each of the spaces in \eqref{eq:inclusions} and \eqref{eq:inclusions-2} is path
connected for $n\ge 4$. If $H_1(M,\Z)\cong\Z^l$, 
there is a surjection from $(\Z_2)^l$ 
to the set of path components of each of the spaces
in \eqref{eq:inclusions} and \eqref{eq:inclusions-2} for $n=3$. 
\end{corollary}

%
%
\begin{problem}
Do the inclusions in \eqref{eq:inclusions}  and \eqref{eq:inclusions-2} induce bijections of path components? Are they weak homotopy equivalences?
\end{problem}

The analogues of Theorems \ref{th:branched},
\ref{th:sing}, and Corollary \ref{cor:main} also hold,
with essentially the same proofs, 
for generalised null curves $F:M\to\C^n$, $n\ge 3$.
These are nonconstant meromorphic maps whose differential
$\di F=dF$ has isolated zeros and poles and assumes 
values in the null quadric $\bA$ \eqref{eq:nullq}. 
Equivalently, given a nowhere vanishing holomorphic $1$-form 
$\theta$ on $M$, we have $dF=f\theta$ where 
$f$ is a holomorphic map from $M$ to the complex subvariety 
$\bA\cup Q\subset\CP^n$ (see \eqref{eq:Q}). 
The real and the imaginary part of any generalised null
curve $M\to\C^n$ are elements of the space
$\CMI_{\br,\s}(M,\R^n)$.

Our method of proof of the main results allows not only to remove the 
singularities but also to move them freely within the surface. 
For example, in the context of Theorem \ref{th:branched}, 
given a closed discrete subset $C \subset M$ and a bijective 
map of $C$ to a subset of the branch locus $\br(u)$, there is an 
isotopy $u_t\in \CMI_\br(M,\R^n)$, $t\in [0,1]$, such that $u_0=u$ 
and $\br(u_1)=C$; see Corollary \ref{co:branched-2}. 
An analogous statement holds in the context
of Theorem \ref{th:sing}; see Corollary \ref{co:sing}.
This follows by a straightforward modification of our proofs.

%
%
%
%
\section{The toolbox}\label{sec:tools}

%
%
\subsection{A Weierstrass interpolation theorem with parameters}

\noindent
In the proof of Theorem \ref{th:branched}, we shall need the following 
parametric version of Weierstrass interpolation for finitely many points in 
an open Riemann surface. This is a special case of 
\cite[Lemma 4.2]{AlarconForstnericLarusson2021GT}
but with added approximation on a compact Runge set.

\begin{lemma}\label{lem:Weierstrass}
Let $K$ be a compact Runge set in an open Riemann surface $M$ and
$a_j:[0,1]\to M\setminus K$, $j=1,\ldots,k$, real analytic maps 
such that the points $a_1(t),\ldots,a_k(t)$ are distinct for every $t\in [0,1]$. 
Also let $\Lambda\subset M$ be a closed discrete subset disjoint from 
$\bigcup_{j=1}^k a_j([0,1])$ and $\lambda:\Lambda\to\N$ a map.
Given integers $n_1,\ldots,n_k\in\Z$ and a number $\epsilon>0$,
there is a real analytic path of meromorphic functions $\{f_t\}_{t\in [0,1]}$ 
on $M$ such that for every $t\in [0,1]$ and $j=1,\ldots,k$, the function 
$f_t$ has degree $n_j$ at $a_j(t)$ and has no other zeros or poles, 
$\max_{x\in K, t\in [0,1]}|f_t(x)-1|<\epsilon$, and $f_t-1$ vanishes to order 
$\lambda(p)$ at $p$ for every $p\in \Lambda$.
 \end{lemma}

\begin{proof}
It suffices to prove the result for $k=n_1=1$. This gives 
for each $j$ a path of holomorphic 
functions $\{f_{j,t}\}_{t\in [0,1]}$ on $M$ with a simple zero at 
$a_j(t)$ and no other zeros, satisfying the approximation 
condition on $K$ and the interpolation
conditions at points $p\in \Lambda$. The function
$f_t=\prod_{j=1}^k f_{j,t}^{n_j}$ then satisfies the theorem. 

The real analytic map $a=a_1:[0,1]\to M$ extends to a holomorphic map 
$a:D\to M$ from an open simply connected neighbourhood $D\subset \C$ 
of the interval $[0,1]\subset \R\subset \C$.
Its graph $\Sigma=\{(z,a(z)):z\in D\} \subset D\times M$ is a 
smooth closed complex hypersurface in the Stein surface $D\times M$. 
Shrinking $D$ around $[0,1]$ if necessary, we ensure that 
$\Sigma\cap (D\times (K\cup \Lambda))=\varnothing$.
Since $D$ is contractible, we have $H^2(D\times M,\Z)\cong H^2(M,\Z)=0$. 
Hence, Oka's solution of the second Cousin problem in \cite{Oka1939} 
implies that every divisor on $D\times M$ is a principal divisor.
Applying this to the divisor $\Sigma$ gives a holomorphic function
$f\in \Oscr(D\times M)$ that vanishes to order $1$ at every 
point of $\Sigma$ and has no other zeros. The function
$f_t=f(t,\cdotp)\in \Oscr(M)$ then has a simple zero at $a(t)$ and
no other zeros for every $t\in D$. 
Since $K$ is Runge in $M$ and $D$ is contractible,
the inclusion $D\times K \hra D\times M$ is homotopy equivalent to
the inclusion of a finite bouquet of circles representing $K$ in the 
finite or countable bouquet of circles representing $M$.
Hence, the map $1/f:D\times K\to \C^*$ extends 
to a continuous map $D\times M\to\C^*$. 
Since $[0,1]\times K$ is holomorphically convex in $D\times M$,
the Oka principle for maps to the complex homogeneous 
manifold $\C^*$ (see Grauert \cite{Grauert1957II} or 
\cite[Theorem 5.4.4]{Forstneric2017E}) gives a 
holomorphic function $g:D\times M\to \C^*$ approximating 
$1/f$ uniformly on $[0,1]\times K$ and such that $g-1/f$ 
vanishes to order $\lambda(p)$ on $D\times \{p\}$ for every $p\in \Lambda$.
(There are no topological obstructions for these
interpolation conditions since the sets $D\times \{p\}$ are contractible.)
Replacing $f$ by $fg$ gives a function satisfying the lemma
provided that the approximation of $1/f$ by $g$ was close enough 
on $[0,1]\times K$. 
\end{proof}

%
%
\subsection{Abelian differentials and complex cones}\label{ss:abelian}

Let $M$ be a Riemann surface. An abelian differential 
$\omega=(\omega_1,\ldots,\omega_n)$ on $M$ with values in $\C^n$ 
(whose components $\omega_i$ are meromorphic 1-forms on $M$) 
is said to be {\em nontrivial} if it is not identically zero, and 
is said to be {\em full} if its range is not contained in a proper
linear subspace of $\C^n$. A nontrivial abelian differential 
determines a divisor $(\omega)$ on $M$ defined as follows. 
Let $\zeta:U\to\C$ be a local holomorphic
coordinate around a point $p\in M$ with $\zeta(p)=0$. 
In this coordinate, $\omega=f(\zeta)d\zeta$ 
where $f=(f_1,\ldots,f_n)$ is a meromorphic map on $U$. 
Let $k(p)\in\Z$ be the unique integer such that 
$\zeta^{-k(p)}f(\zeta)$ is holomorphic near $\zeta=0$ 
and nonvanishing at $\zeta=0$
(i.e., $f_i(0)\ne 0$ for some $i\in \{1,\ldots,n\}$). 
Then, $(\omega)=\sum_{p\in M} k(p) p$. Its support 
$\supp(\omega)=\{p\in M:k(p)\ne 0\}$ 
is a closed discrete subset of $M$. Likewise, the divisors of zeros 
and poles of $\omega$ are, respectively,
\begin{equation}\label{eq:zero-polar}
	(\omega)_0=\sum_{p\in M,\, k(p)>0} k(p)p, \qquad
	(\omega)_\infty=\sum_{p\in M,\, k(p)<0} (-k(p))p,
\end{equation} 
hence $(\omega)=(\omega)_0-(\omega)_\infty$.
The support of $(\omega)_0$ and $(\omega)_\infty$
is the zero set and the polar set of $\omega$, respectively.

%
%
A complex cone in $\C^n$ is a closed analytic subvariety
$A\subset \C^n$ such that $\zeta A\subset A$ for every $\zeta\in \C$.
By a theorem of Chow \cite{Chow1949} (see also 
Chirka \cite[p.\ 74, Remark]{Chirka1989}), such $A$ is the common 
zero set of finitely many homogeneous polynomials on $\C^n$. 
An abelian differential $\omega=(\omega_1,\ldots,\omega_n)$ 
on $M$ is said to have values in $A$ if in any local holomorphic 
coordinate $\zeta$ on $M$ we have $\omega=f(\zeta)d\zeta$,
where $f$ is a meromorphic map with values in $A$.
Such $f$ can be seen as a holomorphic map in the projective
closure of $A$. 

%
%
\subsection{A parametric interpolation theorem for multipliers 
with control of periods}

The following approximation result with interpolation for multiplier functions
is an extension of \cite[Theorem 4.1]{AlarconForstnericLopez2019JGEA}; 
see also \cite[Theorem 5.3.1]{AlarconForstnericLopez2021} and  
\cite[Theorem 2.1]{AlarconLarusson2024complete}.

%
%
\begin{proposition}\label{pr:multipliers}
Assume that $M$ is an open Riemann surface, $K\subset M$ is a 
compact smoothly bounded Runge domain, and 
$\{C_j: j\in I\subset \N\}$ is a collection of smoothly 
embedded oriented Jordan curves in $M$ determining 
a homology basis of $M$ such that 
\begin{itemize}
\item $\bigcup_{j\in J}C_j$ is a Runge compact set in $M$ for every 
finite set $J\subset I$, and 
\item each curve $C_j$ contains a nontrivial arc $\wt C_j$ disjoint 
from $C_i$ for all $i\in I\setminus\{j\}$.
\end{itemize} 
Set $C=\bigcup_{j\in I}C_j$ and $I^K=\{j\in I: C_j\subset K\}$. Let 
$\sigma_a:[0,1]\to K$ $(a=1,\ldots,\alpha\in\N)$ be a finite collection of 
analytic Jordan arcs with pairwise disjoint graphs in $[0,1]\times M$, 
$r\in \N$ an integer, $\Lambda\subset M$ a closed discrete subset, 
and $\lambda:\Lambda\to\N$ a map. 
Set $\Sigma=\bigcup_{a=1}^\alpha \sigma_a([0,1])\subset K$ 
and assume that 
$\Lambda\cap \Sigma=\varnothing=C\cap(\Lambda\cup \Sigma)$.
Let $n\in\N$, let $\theta_t$ $(t\in [0,1])$ be a continuous family of 
$\C^n$-valued full abelian differentials on $M$ with the polar set 
$P_t$ (see Subsect.\ \ref{ss:abelian}), 
set $P=\bigcup_{t\in[0,1]}P_t$, and assume that 
$C\cap P=\varnothing$.  Also let $\varphi_t\in\Oscr(K\cup\Lambda)$ 
$(t\in[0,1])$ be a continuous family of holomorphic functions with 
no zeros on a neighbourhood of $K\cup\Lambda$, and 
$\qgot_j:[0,1]\to\C^n$ $(j\in I)$ a collection of continuous maps 
such that
\[
	\int_{C_j}\varphi_t\theta_t=\qgot_j(t)\quad 
	\text{for every $j\in I^K$ and $t\in [0,1]$}.
\]
Then, the family $\varphi_t$ may be approximated uniformly on 
$[0,1]\times K$ by continuous families of holomorphic functions 
$\wt\varphi_t:M\to\C^*$ $(t\in [0,1])$ satisfying the following conditions:
\begin{enumerate}[\rm (a)]
\item $\int_{C_j}\wt\varphi_t\theta_t=\qgot_j(t)$ for every 
$j\in I$ and $t\in [0,1]$.
\item $\wt\varphi_t-\varphi_t$ vanishes to order $r$ at $\sigma_a(t)$ 
for every $a\in\{1,\ldots,\alpha\}$ and $t\in[0,1]$. 
\item $\wt\varphi_t-\varphi_t$ vanishes to order $\lambda(p)$ at $p$ 
for every $p\in \Lambda$ and $t\in [0,1]$.
\end{enumerate}
Furthermore, if $\varphi_0$ extends to a holomorphic function $M\to\C^*$ 
such that $\int_{C_j}\varphi_0\theta_0=\qgot_j(0)$ for all $j\in I$, then the 
homotopy $\wt \varphi_t$ can be chosen with $\wt\varphi_0=\varphi_0$.
\end{proposition}

The novelties with respect to 
\cite[Theorem 4.1]{AlarconForstnericLopez2019JGEA} are the interpolation
conditions (b) and (c), and the fact that the abelian differentials $\theta_t$ are 
allowed to have poles in the complement of $C$. 
This proposition will be used for various tasks in the proofs of our main results.
In particular, it will be applied in the proof of Theorem \ref{th:sing}
to preserve the residues when moving the poles. 
We shall explain the necessary modifications of 
\cite[proof of Theorem 4.1]{AlarconForstnericLopez2019JGEA} 
which ensure these extra conditions. The 
same arguments apply word by word in the more general framework when 
$K$ is a Runge admissible set 
(see \cite[Def.\ 3.1]{AlarconForstnericLopez2019JGEA}) and the  
multipliers $\varphi_t$ are of class $\Ascr(K)$; this generalisation is well 
understood and we shall not discuss it here. The key to the proof 
of Proposition \ref{pr:multipliers} is the following extension of 
\cite[Lemma 3.2]{AlarconForstnericLopez2019JGEA},
which will also play a crucial role in the proof of Theorem \ref{th:sing}.

\begin{lemma}\label{lem:3.2}
In Proposition \ref{pr:multipliers}, assume in addition 
that $I^K=\{1,\ldots,l\}$ for some $l\in\N$, write 
$C^K=\bigcup_{j=1}^l C_j$, and for each $t\in [0,1]$ let 
$\Pcal^t=(\Pcal^t_1,\ldots,\Pcal^t_l):\Cscr(C^K)\to (\C^n)^l$ denote 
the period map whose $j$-th component $(j=1,\ldots,l)$ is given by
\begin{equation}\label{eq:periodmap}
	\Pcal^t_j(g)=\int_{C_j}g\varphi_t\theta_t,\quad g\in \Cscr(C_j).
\end{equation}
Then there are a convex neighbourhood $D\subset \C$ of 
$[0,1]\subset\R\subset\C$ and a nowhere vanishing holomorphic function 
$\Xi:D\times\C^N\times M\to\C^*$ with $\Xi(\cdot,0,\cdot)\equiv 1$  
satisfying the following conditions:
\begin{enumerate}[\rm (a)]
\item $\Xi(t,\zeta,\cdot)-1$ vanishes to order $r$ at $\sigma_a(t)$ 
for every $a\in \{1,\ldots,\alpha\}$, $t\in [0,1]$, and $\zeta\in\C^N$.
\item $\Xi(t,\zeta,\cdot)-1$ vanishes to order $\lambda(p)$ at $p$ for 
every $p\in\Lambda$, $t\in [0,1]$, and $\zeta\in\C^N$.
\item For every $t\in[0,1]$ the map
\[
	\C^N\ni\zeta\longmapsto \Pcal^t\big(\Xi(t,\zeta,\cdot)\big)\in (\C^n)^l
\]
has maximal rank equal to $ln$ at $\zeta=0$.
\end{enumerate}
\end{lemma}

\begin{proof}
Up to enlarging $K$ slightly, we may assume that 
$\Sigma:=\bigcup_{a=1}^\alpha \sigma_a([0,1])\subset \mathring K$. 
Using \cite[Lemma 2.1]{AlarconForstnericLopez2019JGEA} as in 
\cite[proof of Lemma 3.2]{AlarconForstnericLopez2019JGEA} 
(see also \cite[proof of Proposition 3.1]{AlarconLopez2025AMPA}) we obtain 
for each $j\in\{1,\ldots,l\}$ an integer $N_j\ge n$ and continuous
functions $g_{j,k}:C_j\to\C$ $(k=1,\ldots,N_j)$ with the support on 
the arc $\wt C_j$ such that the function 
$h_j:\C^{N_j}\times C_j\to\C^*$ given by
\[
	h_j(\zeta_j,p)=\prod_{k=1}^{N_j} e^{\zeta_{j,k}g_{j,k}(p)},\quad 
	\zeta_j=(\zeta_{j,1},\ldots,\zeta_{j,N_j})\in\C^{N_j},\quad p\in C_j
\]
satisfies the following period domination condition:
\begin{equation}\label{eq:3.2}
	\frac{\partial}{\partial \zeta_j}\Pcal_j^t\big(h_j(\zeta_j,\cdot)\big) 
	\big|_{\zeta_j=0}:T_0\C^{N_j} \longrightarrow\C^n
	\ \ \text{is surjective for every $t\in [0,1]$}.
\end{equation}
Recall that $\Sigma\cap(C\cup\Lambda)=\varnothing$. Choose a small 
smoothly bounded convex neighbourhood $D\subset\C$ of $[0,1]$ 
such that every analytic arc $\sigma_a:[0,1]\to \mathring K$ 
$(a\in\{1,\ldots,\alpha\})$ extends to a holomorphic map 
$\sigma_a:D\to \mathring K$, and set 
$\Sigma'=\bigcup_{a=1}^\alpha\sigma_a(D)$. 
Let $\delta_a(t)=(t,\sigma_a(t)) \in D\times M$ for $t\in D$.
By choosing the domain $D\supset [0,1]$ small enough,  
we have that $\overline{\Sigma'}\subset\mathring K\setminus (C\cup\Lambda)$ 
and $\{\delta_a(D) : a=1,\ldots,\alpha\}$ is a family of pairwise
disjoint closed complex curves in the Stein surface $D\times M$.
Set $\Delta=\bigcup_{a=1}^\alpha\delta_a(D)\subset D\times \Sigma'$. 
We extend each function $g_{j,k}:C_j\to\C$ $(j\in\{1,\ldots l\}$, 
$k\in\{1,\ldots,N_j\})$ by $0$ to 
$\Sigma'\cup\Lambda\cup(C^K\setminus C_j)$
and view it as a continuous map 
$g_{j,k}:D\times (\Sigma'\cup\Lambda\cup C^K)\to\C$ given by 
$g_{j,k}(t,\cdot)=g_{j,k}$ for all $t\in D$.  
Note that $g_{j,k}$ vanishes on $\Delta\cup(D\times\Lambda)$. 
It is clear that $g_{j,k}$ extends to a continuous function on $D\times M$ 
that vanishes on a neighbourhood of 
the divisor $\Delta\cup (D\times\Lambda)$. 
Since the compact set $[0,1]\times C^K\subset D\times M$
is holomorphically convex and $C^K$ is a union of curves, 
Mergelyan's theorem shows that we can 
approximate $g_{j,k}$ uniformly on $[0,1]\times C^K$ 
by a holomorphic function on a neighbourhood of $[0,1]\times C^K$ 
in $D\times M$, which we still denote $g_{j,k}$.
Next, a standard recursive application of the Oka--Weil theorem 
with jet interpolation enables us to approximate $g_{j,k}$ uniformly on 
$[0,1]\times C^K$ by a holomorphic function $\wt g_{j,k}\in \Oscr(D\times M)$
vanishing to any given order on each connected component of 
$\Delta\cup (D\times\Lambda)$.
(These components are $\delta_a(D)$ $(a=1,\ldots,\alpha)$
and $D\times \{p\}$ for $p\in \Lambda$.) 
In particular, $\wt g_{j,k}$ can be chosen such that 
$\wt g_{j,k}(t,\cdot)$, $t\in [0,1]$, vanishes to order $r$ at the point 
$\sigma_a(t)$ for all $a\in\{1,\ldots,\alpha\}$, 
and it vanishes to order $\lambda (p)$ 
at every point $p\in \Lambda$. Set
\[
	\Xi(t,\zeta,p)=\prod_{j=1}^l\prod_{k=1}^{N_j}e^{\zeta_{j,k}\wt g_{j,k}(t,p)}, 
	\quad t\in D,\; \zeta=(\zeta_1,\ldots,\zeta_l)\in 
	\C^{N_1}\times\cdots\times\C^{N_l},\; p\in M. 
\]
Setting $N=\sum_{j=1}^l N_j\ge nl$ and identifying 
$\C^N=\C^{N_1}\times\cdots\times\C^{N_l}$, it is clear that 
$\Xi: D\times\C^N\times M\to\C^*$ is holomorphic and satisfies 
$\Xi(\cdot,0,\cdot)\equiv 1$ and conditions (a) and (b) in the lemma. 
Moreover, (c) is guaranteed by \eqref{eq:3.2} whenever the
approximation of each $g_{j,k}$ by $\wt g_{j,k}$ on $[0,1]\times C^K$ 
is close enough.
\end{proof}

%
%
\begin{proof}[Proof of Proposition \ref{pr:multipliers}]
Choose a normal exhaustion
\[
	K=K_0\subset K_1\subset K_2\subset \cdots\subset 
	\bigcup_{i=0}^\infty K_i=M
\]
by smoothly bounded Runge compact domains such that, 
setting $I^i=\{j\in I: C_j\subset K_i\}$ for $i=0,1,2,\ldots$ 
(note that $I^i$ is finite and $I^i\subset I^{i+1}$ for every $i\ge 0$), 
the following conditions hold for every $i\in \N$:
\begin{itemize}
\item $I^i\setminus I^{i-1}$ is either empty or a singleton.
\item The compact set $K_{i-1}\cup\bigcup_{j\in I^i}C_j$ is Runge 
in $M$ and admissible in the sense of  
\cite[Def.\ 3.1]{AlarconForstnericLopez2019JGEA}.
\end{itemize}
In order to ensure the latter condition for $i=1$ we might need to 
replace $K=K_0$ by a slightly larger compact domain.
Set $\varphi_t^0=\varphi_t:K_0\cup\Lambda\to\C^*$ for $t\in [0,1]$. 
The proof consists of constructing a sequence of continuous families 
$\{\varphi_t^i\in\Oscr(K_i\cup\Lambda)\}_{i\in\N}$ $(t\in [0,1])$ of 
holomorphic functions without zeros on a neighbourhood 
of $K_i\cup\Lambda$ such that the following conditions hold for 
all $t\in[0,1]$ and $i\in\N$:
\begin{enumerate}[\rm (A{$_i$})] 
\item $\varphi_t^i$ is as close as desired to $\varphi_t^{i-1}$ 
uniformly on $[0,1]\times K_{i-1}$.
\item $\int_{C_j} \varphi_t^i\theta_t=\qgot_j(t)$ holds for every $j\in I^i$.
\item $\varphi_t^i-\varphi_t$ vanishes to order $r$ at the point 
$\sigma_a(t)$ for every $a\in\{1,\ldots,\alpha\}$ and $t\in[0,1]$. 
\item $\varphi_t^i-\varphi_t$ vanishes to order $\lambda(p)$ at  
every point $p\in \Lambda$.
\item If $\varphi_0$ extends to a holomorphic function $M\to\C^*$ such that 
$\int_{C_j}\varphi_0\theta_0=\qgot_j(0)$ for all $j\in I$, then the homotopy 
$\varphi_t^i$ can be chosen with $\varphi_0^i=\varphi_0$.
\end{enumerate}
As in the proof of \cite[Theorem 4.1]{AlarconForstnericLopez2019JGEA}, if the approximation in (A$_i$) is close enough for every $i\in\N$, we obtain a limit continuous family of holomorphic functions
$\wt\varphi_t=\lim_{i\to\infty}\varphi_t^i:M\to\C^*$, $t\in [0,1]$, 
satisfying Proposition \ref{pr:multipliers}. 
Conditions (b) and (c) are trivially guaranteed by (C$_i$) and (D$_i$). 

We proceed by induction. The base is given by the family $\varphi_t^0$ $(t\in [0,1])$. For the inductive step, we assume that we have a suitable family $\varphi_t^{i-1}$ for some $i\in \N$ and will provide $\varphi_t^i$. 
We distinguish cases.

{\em The noncritical case: $I^i= I^{i-1}$.} Assume that 
$I^i=\{1,\ldots,l\in\N\}\neq\varnothing$, for the proof is much simpler otherwise.
Set $C^i=\bigcup_{j=1}^lC_j$ and for each $t\in [0,1]$ consider the period map 
$\Pcal^t:\Cscr(C^i)\to(\C^n)^l$ defined by \eqref{eq:periodmap} 
with $\varphi_t$ replaced by $\varphi_t^{i-1}$. By Lemma \ref{lem:3.2}, 
there are a convex neighbourhood $D\subset \C$ of $[0,1]\subset\R\subset\C$ 
and a nowhere vanishing holomorphic function 
$\Xi:D\times\C^N\times M\to\C^*$ such that $\Xi(\cdot,0,\cdot)\equiv 1$ 
and the following conditions hold for every $t\in [0,1]$. 
\begin{enumerate}[\rm (I)]
\item $\Xi(t,\zeta,\cdot)-1$ vanishes to order $r$ at $\sigma_a(t)$ for every 
$a\in \{1,\ldots,\alpha\}$ and $\zeta\in\C^N$.
\item $\Xi(t,\zeta,\cdot)-1$ vanishes to order $\lambda(p)$ at $p$ for 
every $p\in\Lambda$ and $\zeta\in\C^N$.
\item The map 
$\C^N\ni\zeta\longmapsto \Pcal^t\big(\Xi(t,\zeta,\cdot)\big)\in (\C^n)^l$ 
has maximal rank equal to $ln$ at $\zeta=0$.
\end{enumerate}
Taking into account conditions (B$_{i-1}$)--(E$_{i-1}$) and using a similar 
argument as in the proof of Lemma \ref{lem:3.2} to ensure parametric 
interpolation, we find a continuous family of holomorphic functions
$\phi_t:K_i\cup\Lambda\to\C^*$, $t\in [0,1]$, on a neighbourhood 
of $K_i\cup\Lambda$  
satisfying the following conditions: 
\begin{enumerate}[\rm (i)]
\item $\phi_t$ is as close as desired to $\varphi_t^{i-1}$ uniformly 
on $[0,1]\times K_{i-1}$.
\item $\phi_t-\varphi_t$ vanishes to order $r$ at $\sigma_a(t)$ for every
$a\in\{1,\ldots,\alpha\}$ and $t\in[0,1]$.
\item $\phi_t-\varphi_t$ vanishes to order $\lambda(p)$ at $p$ for 
every $p\in \Lambda$ and $t\in [0,1]$.
\item If $\varphi_0$ extends to a holomorphic function $M\to\C^*$ 
such that $\int_{C_j}\varphi_0\theta_0=\qgot_j(0)$ for all $j\in I$, 
then the homotopy $\phi_t$ can be chosen such that $\phi_0=\varphi_0$.
\end{enumerate}
In view of  condition (III), if the approximation in (i) is close enough then, 
arguing as in \cite[proof of Lemma 4.2]{AlarconForstnericLopez2019JGEA}, 
the implicit function theorem furnishes a continuous path $\beta:[0,1]\to\C^N$ 
such that $\Xi(t,\beta(t),\cdot)$ is close to $1$ uniformly on $K_{i-1}$
for all $t\in [0,1]$ and the continuous family of holomorphic functions
\[
	\varphi_t^i:=\Xi(t,\beta(t),\cdot)\phi_t: K_i\cup\Lambda \to\C^*,
	\quad t\in [0,1]
\] 
satisfies conditions (A$_i$)--(E$_i$);
in particular, in the assumptions in (iv) we can choose $\beta$ with 
$\beta(0)=0$. Note that (C$_i$) is ensured by (C$_{i-1}$), (I), and (ii), 
while (D$_i$) is guaranteed by (D$_{i-1}$), (II), and (iii).

{\em The critical case: $I^i\neq I^{i-1}$.} 
In this case $I^i\setminus I^{i-1}=\{j\}\subset I$.
Taking into account that $C_j\cap(\Lambda\cup \Sigma\cup P)=\varnothing$ 
and $K_{i-1}\cup C_j$ is an admissible Runge compact set in $M$, the 
construction is reduced to the noncritical case by using Lemma \ref{lem:3.2} 
and \cite[Lemma 2.3]{AlarconForstnericLopez2019JGEA}. 
The details are similar to
\cite[proof of Lemma 4.3]{AlarconForstnericLopez2019JGEA}
and we leave them out.
This completes the proof of Proposition \ref{pr:multipliers}.
\end{proof}

%
%
%
\section{Removing branch points}\label{sec:branched}

\noindent
In this section we establish the following extension
of Theorem \ref{th:branched} which says that we can move
the branch points out of a minimal surface while keeping the poles fixed. 
This result and Theorem \ref{th:sing} (on moving the poles)
trivially imply Corollary \ref{cor:main}. We shall use the notation 
in \eqref{eq:GaussMap} and \eqref{eq:Ecal}.

%
%
%
\begin{theorem}\label{th:branched-2}
Let $M$ be an open Riemann surface and $n\ge 3$ an integer.
For every $u\in \CMI_{\br,\s}(M,\R^n)$ there is an isotopy 
$u_t \in \CMI_{\br,\s}(M,\R^n)$, $t\in [0,1]$, such that $u_0=u$, 
$u_t-u$ is continuous on $M$ (hence $\Ecal_{u_t}^\infty=\Ecal_u^\infty$) 
for all $t\in [0,1]$, and $u_1:M\setminus \Ecal_{u_1}^\infty\to\R^n$ is 
unbranched, hence $u_1\in\CMI_\s(M,\R^n)$. Furthermore, we can 
choose the isotopy such that for each $t\in [0,1]$ the Gauss map 
$\Gscr(u_t)$ of $u_t$ equals $\Gscr(u)$ in their common domain of definition 
$M\setminus(\br(u_t)\cup\br(u)\cup \Ecal_u^\infty)$.
\end{theorem}

The theorem says in particular that the inclusion 
$\CMI_\s(M,\R^n)\hra \CMI_{\br,s}(M,\R^n)$ induces a surjection of 
path components. In the proof, we shall need the following result. The notion 
of an abelian differential with values in a complex cone $A\subset\C^n$ 
was introduced in Subsect.\ \ref{ss:abelian}.

%
%
\begin{proposition}\label{prop:abelian}
Assume that $M$ is a connected open Riemann surface, $A\subset \C^n$ 
is a closed complex cone of positive dimension, and $\omega$ is a nontrivial 
abelian differential on $M$ with values in $A$. Then there is a path of nontrivial abelian differentials $\omega_t=h_t\omega$ $(t\in [0,1])$ with values in $A$, 
where $h_t$ is a path of meromorphic functions on $M$, satisfying $\omega_0=\omega$ (that is, $h_0\equiv1$) and either of the following 
conditions:
\begin{enumerate}[\rm (a)]
\item  $\omega_1 = h_1\omega$ is a holomorphic $1$-form on $M$ 
without zeros.
\item $h_t$ has no zeros for every $t\in [0,1]$, 
$\omega_t-\omega=(h_t-1)\omega$ is holomorphic on $M$ for every 
$t\in [0,1]$, and $\omega_1 = h_1\omega$ is an abelian differential 
on $M$ without zeros.
\end{enumerate}
\end{proposition}

\begin{proof}
Write $(\omega) = \sum_{p\in M} k(p) p =  (\omega)_0-(\omega)_\infty$
(see Subsect.\ \ref{ss:abelian}).

We first explain how to obtain a path $\omega_t$ $(t\in [0,1])$ as in the 
statement satisfying condition (a). Choose a normal exhaustion 
$K_0\subset K_1\subset\cdots\subset \bigcup_{i=0}^\infty K_i=M$
by compact Runge sets such that 
$\supp(\omega)\cap K_0=\varnothing$.
Let $\supp(\omega) \cap K_1=\{p_1,\ldots,p_m\}$ and set 
$n_j=k(p_j)$ for $j=1,\ldots,m$. For every $j=1,\ldots,m$
we choose a real analytic path  $a_j:[0,1]\to M\setminus K_0$ 
such that $a_j(0)=p_j$ and $a_j(1)\in M\setminus K_1$.
Pick a number $\epsilon_1>0$.
Lemma \ref{lem:Weierstrass} furnishes a path
$\{f^1_t\}_{t\in [0,1]}$ of meromorphic functions on $M$ with divisors  
\begin{equation}\label{eq:lem2.1.1}
	(f^1_t)=\sum_{j=1}^m n_j a_j(t),\quad t\in [0,1],
\end{equation} 
such that 
\begin{equation}\label{eq:lem2.1.2}
	\max_{x\in K_0, \, t\in [0,1]}|f^1_t(x)-1|<\epsilon_1.
\end{equation}
Note that $a_j(0)=p_j$, so the divisor 
$(f^1_0)=\sum_{j=1}^m n_j p_j$ is precisely the 
part of the divisor $(\omega)$ lying in $K_1$. Hence, the $1$-form
$\omega'=\frac{1}{f^1_0}\omega$ has no zeros or poles on $K_1$.
Consider the path of abelian differentials 
\begin{equation}\label{eq:lem2.1.3}
	\omega_t = \frac{f^1_t}{f^1_0} \omega=f^1_t\omega',
	\quad t\in [0,1].
\end{equation}
We have that $\omega_0=\omega$, $\omega_1=f^1_1\omega'$, 
$\supp(\omega_t) \cap K_0=\varnothing$ for all $t\in [0,1]$,
$\supp(\omega_1)\cap K_1=\varnothing$, and 
$\omega_t$ approximates $\omega_0$ on $K_0$ for all $t\in [0,1]$.
We now repeat the same procedure with the abelian differential 
$\omega_1$ in order to find a path of meromorphic functions 
$\{f^2_t\}_{t\in [1,2]}$ on $M$ such that the divisor $(f^2_1)$ agrees 
with the part of the divisor $(\omega_1)$ on $K_2$,
$\supp (f^2_t)\cap K_1=\varnothing$ for all $t\in [1,2]$, 
$\supp (f^2_2)\subset M\setminus K_2$, and 
\[
	\max_{x\in K_1,\, t\in [1,2]}|f^2_t(x)-1|<\epsilon_2
\] 
for a given $\epsilon_2>0$. Set 
\[
	\omega_t = \frac{f^2_t}{f^2_0} \omega_1,\quad t\in [1,2].
\]
Then, $\supp(\omega_t) \cap K_1=\varnothing$ for all $t\in [1,2]$,
$\supp(\omega_2) \cap K_2 =\varnothing$, and $\omega_t$
approximates $\omega_1$ uniformly on $K_1$ for all $t\in [1,2]$.
Continuing inductively, we obtain a path of abelian differentials
$\omega_t$, $t\in [0,\infty)$, such that 
$\supp(\omega_t)\cap K_j =\varnothing$ for all $t\ge j$ and $j=0,1,\ldots$.
Choosing $\epsilon_j>0$ small enough at every step,
the approximation conditions ensure that 
$\omega_\infty=\lim_{t\to\infty}\omega_t$ 
is an abelian differential without zeros or poles on $M$. 
It remains to reparametrise the interval $[0,\infty]$ to $[0,1]$.
This explains part (a) of the proposition.

We now explain how to modify the above argument to obtain 
a path of abelian differentials 
$\omega_t$, $t\in [0,1]$, satisfying condition (b). 
Set $\Lambda=\supp(\omega)_\infty$ and write 
$(\omega)_\infty=\sum_{p\in\Lambda}\lambda(p)p$. Choose 
$K_0\subset K_1\subset\cdots$ as above, let 
$\supp(\omega)_0 \cap K_1=\{p_1,\ldots,p_m\}$ and set $n_j=k(p_j)>0$ 
for $j=1,\ldots,m$. Choose analytic paths 
$a_j:[0,1]\to M\setminus K_0$ $(j=1,\ldots,m)$ 
such that $a_j(0)=p_j$, $a_j(1)\in M\setminus K_1$, and 
$\Lambda\cap\bigcup_{j=1}^m a_j([0,1])=\varnothing$. 
Lemma \ref{lem:Weierstrass} provides a path
$\{f^1_t\}_{t\in [0,1]}$ of holomorphic functions on $M$ satisfying \eqref{eq:lem2.1.1}, \eqref{eq:lem2.1.2}, and
\begin{equation}\label{eq:ft1-1}
	\text{$f_t^1-1$ vanishes to order $\lambda(p)$ at $p$ 
	for every $p\in \Lambda$.}
\end{equation}
Note that $(f_0^1)=\sum_{j=1}^m n_jp_j$ is the part
of $(\omega)_0$ lying in $K_1$, 
$\omega'=\frac1{f_0^1}\omega$ has no zeros on $K_1$, and the path 
of abelian differentials $\omega_t=f^1_t\omega'$ $(t\in [0,1])$ defined as in \eqref{eq:lem2.1.3} satisfies $\omega_0=\omega$, 
$\supp(\omega_t)_0 \cap K_0=\varnothing$ for all $t\in [0,1]$,
$\supp(\omega_1)_0\cap K_1=\varnothing$, 
$\omega_t$ approximates $\omega_0$ on $K_0$ for all $t\in [0,1]$, and
the difference $\omega_t-\omega$ is holomorphic for all $t\in [0,1]$ 
as guaranteed by \eqref{eq:ft1-1}.  
Repeating the same procedure in a recursive way as above leads to a path 
of abelian differentials $\omega_t$ $(t\in [0,\infty])$ that, after reparametrising 
$[0,\infty]$ to $[0,1]$, satisfies the required properties.
\end{proof}

\begin{proof}[Proof of Theorem \ref{th:branched-2}]
Let $u_0\in \CMI_{\br,\s}(M,\R^n)$.
Then, $\omega_0=2\di u_0$ is an abelian differential on $M$ 
with values in the null quadric  $\bA$ \eqref{eq:nullq} whose divisor satisfies 
$\supp (\omega_0)=\Ecal_{u_0}^0\cup \Ecal_{u_0}^\infty=\br (u_0)\cup P(u_0)$
(see \eqref{eq:Ecal} and Subsect.\ \ref{ss:abelian} for the notation). 
Let $\{\omega_t\}_{t\in [0,1]}$ be a path of nontrivial abelian differentials 
on $M$ with values in $\bA$, provided by Proposition \ref{prop:abelian} (b), 
so $\omega_t=h_t\omega_0$ for some meromorphic function $h_t$ on $M$ 
and $\omega_t-\omega_0$ is holomorphic for all $t\in [0,1]$,
and $\omega_1$ has no zeros. In particular, 
$(\omega_t)_\infty=(\omega_0)_\infty$ for all $t\in [0,1]$.
Let $\Lambda=\supp(\omega_0)_\infty$ and write 
$(\omega_0)_\infty=\sum_{p\in\Lambda}\lambda(p)p$. 
Let $\{C_j: j\in I\subset \N\}$ be a collection of smoothly embedded oriented 
Jordan curves in $M$ determining a homology basis of $M$ such that 
$\bigcup_{j\in J}C_j$ is Runge in $M$ for every finite set $J\subset I$, 
each curve $C_j$ contains a nontrivial arc $\wt C_j$ disjoint 
from $C_i$ for all $i\in I\setminus \{j\}$, and $C_j\cap\Lambda=\varnothing$ 
for all $j\in I$. 
Since the real part $\Re(\omega_0)$ of $\omega_0$ is exact on $M$, 
we have that $\Re\int_{C_j}\omega_0=0$ for every $j\in I$.
Proposition \ref{pr:multipliers} then furnishes 
a path  $\{\xi_t\}_{t\in [0,1]} \subset \Oscr^*(M)$ of nowhere vanishing holomorphic functions on $M$, with $\xi_0=1$, such that 
\begin{equation}\label{eq:xit-1}
	\text{$\xi_t-1$ vanishes to order $\lambda(p)$ at $p$ for every 
	$p\in \Lambda$ and $t\in [0,1]$}, 
\end{equation}
and 
\begin{equation}\label{eq:Rexit}
	\text{$\Re\int_{C_j}\xi_t\omega_t=0$\ \ for every $j\in I$ and $t\in [0,1]$.}
\end{equation} 
Since $(\omega_t)_\infty=\sum_{p\in\Lambda}\lambda(p)p$, condition 
\eqref{eq:xit-1} ensures that $\xi_t\omega_t-\omega_t$ is holomorphic 
on $M$, and hence so is $\xi_t\omega_t-\omega_0$ for every $t\in [0,1]$
(recall that $\omega_t-\omega_0$ is holomorphic on $M$). Thus, taking into account that the curves $C_j$ $(j\in I)$ are a homology basis of $M$ and 
$\Re(\omega_0)$ is exact on $M$, \eqref{eq:Rexit} implies 
that $\Re(\xi_t\omega_t)$ is exact on $M$ as well for every $t\in [0,1]$.
Therefore, the real parts of the abelian differentials $\xi_t\omega_t$ 
integrate by the Weierstrass formula \eqref{eq:W-formula} to a path of 
conformal minimal surfaces $u_t\in \CMI_{\br,\s}(M,\R^n)$ $(t\in [0,1])$ 
such that $u_0$ is the given initial map and for every $t\in [0,1]$
we have $\br(u_t)=\supp (\xi_t\omega_t)_0=\supp (\omega_t)_0$, 
$P(u_t)=\supp (\xi_t\omega_t)_\infty=\supp (\omega_t)_\infty=P(u_0)$, 
and $u_t-u$ is continuous on $M$; recall that each $\xi_t$ has neither 
zeros nor poles. In particular, $\br(u_1)=\supp(\omega_1)_0=\varnothing$, 
and hence $u_1\in \CMI_{\s}(M,\R^n)$.
Finally, since $2\partial u_t=\xi_t\omega_t=2\xi_t h_t \partial u_0$, 
it is clear that the Gauss map $\Gscr(u_t)$ of $u_t$ equals  
$\Gscr(u_0)$ on $M\setminus(\br(u_t)\cup\br(u_0)\cup P(u_0))$ 
for every $t\in [0,1]$; see \eqref{eq:GaussMap}. 
\end{proof}

\begin{remark}\label{rem:nofull}
A comment is in order regarding the use of Proposition \ref{pr:multipliers} 
in the proof of Theorem \ref{th:branched-2}.
Proposition \ref{pr:multipliers} is stated for a path of 
full abelian differentials $\theta_t$ on $M$. 
In our situation, $\omega_t$ takes values in the 
null quadric $\bA$ \eqref{eq:nullq} and is not assumed to be full. 
However, there is a $\C$-linear subspace $H\subset\C^n$ 
such that $\omega_0$ is full in $H$ (meaning that $\omega_0/\theta:M\to H$ 
is full for any nowhere vanishing holomorphic $1$-form $\theta$ on $M$), 
and the same is then true for every $\omega_t=h_t\omega_0$ $(t\in [0,1])$ 
in the path given by the proposition. Applying Proposition \ref{pr:multipliers} 
to the path $\omega_t$ with values in $H$ gives a family of multipliers 
$\{\xi_t\}_{t\in [0,1]} \subset \Oscr^*(M)$ with the properties stated in 
the proof of Theorem \ref{th:branched-2}. 
\end{remark}

Let us record here the following extension of Theorem \ref{th:branched-2} 
which follows by a straightforward modification of the proof.

\begin{corollary}\label{co:branched-2}
If $M$, $n$, and $u$ are as in Theorem \ref{th:branched-2} and 
$C\subset M\setminus  \Ecal_u^\infty$ is a (possibly empty) closed discrete 
subset of $M$ that is in bijection with a subset of $\br(u)$, then there is an isotopy $u_t \in \CMI_{\br,\s}(M,\R^n)$, $t\in [0,1]$, such that $u_0=u$, 
$u_t-u$ is continuous on $M$ for all $t\in [0,1]$, and $\br(u_1)=C$. 
Furthermore, we can choose the isotopy such that for each $t\in [0,1]$ the 
Gauss map $\Gscr(u_t)$ of $u_t$ equals $\Gscr(u)$ in 
their common domain of definition 
$M\setminus(\br(u_t)\cup\br(u)\cup \Ecal_u^\infty)$.
\end{corollary}

%
%
\section{Removing complete ends of finite total curvature}\label{sec:FTC}

\noindent
In this section, we prove Theorem \ref{th:sing}. In view of 
Theorem \ref{th:branched-2}, this will also yield Corollary \ref{cor:main}.

Let $M$ be an open Riemann surface and $u\in \CMI_\s^\f(M,\R^n)$
for some $n\ge 3$. Denote by $P=P(u) =\{p_j\}_j \subset M$ 
the closed discrete set of poles of the abelian differential $\di u$. 
We shall construct a path $P_t=\{p_j(t)\}_j \subset M$
$(t\in [0,1))$ of closed discrete subsets such that the graphs 
of the paths $p_j(t)$ in $[0,1]\times M$ are pairwise disjoint, 
$p_j(0)=p_j$ and $p_j(t)$ diverges to infinity in $M$ as $t\to 1$ 
for every $j$, and an isotopy $u_t:M\setminus P_t\to\R^n$ of conformal
minimal immersions with a complete end of finite total curvature at every
point of $P_t,\ t\in [0,1)$ (that is, $P(u_t)=P_t$) 
such that the limit $u_1=\lim_{t\to 1} u_t:M\to \R^n$ exists and 
is a conformal minimal immersion without singularities.

Fix a holomorphic immersion $z:M\to \C$, which therefore
provides a local holomorphic coordinate on $M$ 
on a neighbourhood of any point.
Choose a normal exhaustion $K_0\subset K_1\subset K_2\subset \cdots$ 
of $M$ by smoothly bounded compact Runge sets, 
each contained in the interior of the
next one, such that $P\cap K_0=\varnothing$. 
We shall proceed inductively, using the parameter 
interval $[i,i+1]\subset \R$ in the $i$-th step of the induction 
and finally reparametrising $[0,+\infty]$ 
to $[0,1]$ as in the proof of Proposition \ref{prop:abelian}. 

We begin by explaining the initial step of the construction with $i=0$; 
every subsequent step will be of the same kind. 
Choose real analytic paths $p_j(t)\in M\setminus K_0$, 
$t\in [0,1]$, with pairwise disjoint graphs in $[0,1]\times M$
such that $p_j(0)=p_j$ and $p_j(1)\in M\setminus K_1$ for all $j=1,2,\ldots$. 
The path $p_j(t)$ is chosen to be independent of $t$ if $p_j\in M\setminus K_1$, 
which holds for all but finitely many $j$. Let $\gamma_j(t)=(t,p_j(t))$ for 
$t\in [0,1]$. There is a convex neighbourhood $D\subset \C$ of $[0,1]$
such that every $p_j$ extends from $[0,1]$ to a holomorphic map
$p_j:D\to M$, and $\Gamma_j=\gamma_j(D)$ is a family of pairwise disjoint 
closed complex curves (graphs of $p_j$ over $D$) in the 
Stein surface $D\times M$. Let $V_j$ denote the holomorphic vector field 
on $D\times M$ given by 
\[
	V_j=\frac{\di}{\di t} + \dot p_j(t)\frac{\di}{\di z}\Big|_{x},
\]
where $t$ is the coordinate on $\C$ and $z:M\to \C$ is the holomorphic
immersion. Note that $V_j$ is tangent to the complex curve $\Gamma_j$.
Let $\phi_{j,t}(x)=(t,\varphi_{j,t}(x))$ denote the holomorphic 
flow of $V_j$ satisfying the initial condition $\varphi_{j,0}(x)=x$. 
Shrinking $D$ around $[0,1]$ if necessary, the flow is defined 
for all $t\in D$ and all $x$ in a disc neighbourhood $U_j\subset M$ 
of the point $p_j=p_j(0)$ for every $j$, and the map 
\[
	\varphi_{j,t}: U_j\to U_{j,t}:=\varphi_{j,t}(U_j) \subset M
\] 
is biholomorphic and satisfies $\varphi_{j,t}(p_j)=p_j(t)$ 
for every $t\in D$ and every $j$.  
Note that for all but finitely many $j$ we have $V_j=\frac{\di}{\di t}$
and hence $\varphi_{j,t}$ is the identity on $U_j$ for all $t\in D$. 
Choosing the discs $U_j$ small enough, we may assume 
that the sets $\{U_{j,t}\}_j$ are pairwise disjoint for all $t\in D$,
and they are also disjoint from a neighbourhood $W_0\subset M$
of $K_0$. Write $2\di u_0=f_0\theta$. 
Consider the path of abelian differentials $\omega_t$ on the 
domains $\wt U_t:=W_0\cup \bigcup_{j}U_{j,t} \subset M$, $t\in D$, 
defined by 
\begin{equation}\label{eq:omegatdef}
	\omega_t = 
	\begin{cases}
                    (\varphi_{j,t}^{-1})^* (f_0\theta)    
                    &  \text{on $U_{j,t}$}, \\
                    f_0\theta & \text{on $W_0$}. 
        \end{cases}
\end{equation}
Note that $\omega_t$ is full, it assumes values in the null quadric 
$\bA_*$ \eqref{eq:nullq} for every $t\in D$,   
it has a pole of order $n_j\ge 2$ at the point $p_j(t)$ 
for every $j$ (with $n_j$ independent of $t$), and 
it has no other zeros or poles. Furthermore, the residue of 
$\omega_{t}$ at $p_j(t)$ is independent of $t\in D$, so it 
has vanishing real part (as this holds for $\omega_0=2\di u_0$). 
Write $\omega_t=f_t\theta$ for $t\in D$. 
Then, $f_t:\wt U_t \cup W_0\to Y=\bA_*\cup Q$ 
is a holomorphic map depending holomorphically on $t\in D$,
and $f_0=2\di u_0/\theta$ is holomorphic on all of $M$. 
(Recall that $Q\subset \CP^{n-1}$ is the hyperquadric \eqref{eq:Q}.) 
Writing $f(t,x)=f_t(x)$, the map $f$ with values in $Y$
is holomorphic on the open set 
\[
	O:= (D\times W_0) \cup \bigcup_{t\in D} (\{t\}\times \wt U_t)  
	\subset D\times M
\]
and on the complex submanifold $\{0\}\times M$. Note that 
$f^{-1}(Q)=\Gamma:=\bigcup_{j} \Gamma_j$ is a complex
submanifold contained in $O$. 
It is obvious that $f$ extends from a somewhat smaller
open set containing $(D\times K_0) \cup \Gamma$ to a continuous 
map $f:D\times M\to Y$ which agrees with $f_0$ on $\{0\}\times M$ 
and maps $(D\times M)\setminus \Gamma$ to $\bA_*$. 

Let $\{C_l: l\in I\subset\N\}$ be a collection of smooth oriented Jordan curves in
$M$ determining a homology basis of $M$ such that $\bigcup_{l\in J} C_l$ is a 
Runge compact set in $M$ for every finite set $J\subset I$, and each curve 
$C_l$ contains a nontrivial arc $\wt C_l$ that is disjoint form $C_i$ for all 
$i\in I\setminus\{l\}$. In addition, we choose these curves such that 
$C\cap p_j([0,1])=\varnothing$ for every $j=1,2,\ldots$ where 
$C=\bigcup_{l\in I}C_l$, and $C_1,\ldots,C_\ell$ determine a homology basis 
of $K_0$. We assume that $\ell>0$ since the proof is much simpler otherwise.
Let $\Pcal^t=(\Pcal^t_1,\ldots,\Pcal^t_\ell):\Cscr(C^0)\to(\C^n)^\ell$ 
be the period map whose $l$-th component is given by
\[
	\Pcal^t_l(g)=\int_{C_l}g\omega_t,\quad g\in \Cscr(C_l).
\]
Up to shrinking the domain $D\supset [0,1]$, 
Lemma \ref{lem:3.2} furnishes a nowhere vanishing holomorphic function 
$\Xi:D\times\C^N\times M\to\C^*$ satisfying the following conditions:
\begin{enumerate}[\rm (I)]
\item $\Xi(t,\zeta,\cdot)-1$ vanishes to order $n_j$ at $p_j(t)$ for every 
$t\in [0,1]$, $\zeta\in\C^N$, and $j=1,2,\ldots$. 
\item The map 
$\C^N\ni\zeta\longmapsto \Pcal^t\big(\Xi(t,\zeta,\cdot)\big)\in (\C^n)^\ell$ 
has maximal rank at $\zeta=0$ for each $t\in [0,1]$.
\item $\Xi(\cdot,0,\cdot)\equiv 1$.
\end{enumerate}

Recall that the manifolds $Q\subset \CP^{n-1}$ \eqref{eq:Q}, 
$Y=\bA_*\cup Q\subset \CP^n$ \eqref{eq:Y}, and $\bA_*=Y\setminus Q$ 
are Oka manifolds. Indeed, $\bA_*$ is a homogeneous space
of the complex Lie group $O(n,\C)$, and hence an Oka manifold
by Grauert's theorem \cite{Grauert1958MA}. 
(See also \cite[Proposition 5.6.1 and Example 5.6.2]{Forstneric2017E}.) 
The projection $\pi:\bA_*\to Q$, 
$\pi(z_1,\ldots,z_n)= [z_1:\cdots:z_n]$, is a holomorphic
fibre bundle with Oka fibre $\C^*=\C\setminus \{0\}$, so 
$Q$ is Oka by \cite[Theorem 5.6.5]{Forstneric2017E}.
Finally, $\pi: Y\to Q$ is a holomorphic line bundle, so 
$Y$ is Oka by the same theorem.

Recall that $n_j\ge 2$ denotes the order of the pole of 
$f_0=2\di u_0/\theta$ at the point $p_j$ for $j=1,2,\ldots$. 
By the Oka principle (see \cite[Theorem 5.4.4]{Forstneric2017E}), 
there is a holomorphic map $F:D\times M\to Y$ 
which agrees with $f$ to order $n_j$ along $\Gamma_j$
for every $j$, it agrees with $f(0,\cdotp)$ on $\{0\}\times M$,
and it approximates $f$ as closely as desired uniformly on 
$[0,1]\times K_0$. Moreover, after shrinking
$D$ around $[0,1]$ we can ensure that
$F^{-1}(Q)=f^{-1}(Q)=\Gamma$, that is, $F$ maps 
$(D\times M)\setminus \Gamma$ to the Oka 
domain $\bA_*=Y\setminus Q\subset Y$.
This can be obtained by inductively using  
\cite[Theorem 1.3]{Forstneric2023Indag}. 
Then, $\theta_t=F(t,\cdotp) \theta$ $(t\in [0,1])$ 
is an analytic path of abelian differentials on $M$ 
with values in $\bA_*$ satisfying the following conditions.
\begin{enumerate}[\rm (i)] 
\item $\theta_0=\omega_0=2\di u_0$.
\item $\theta_t-\omega_t$ is holomorphic near $p_j(t)$
for every $t\in [0,1]$ and $j=1,2,\ldots$ 
\item $\theta_t$ approximates $\omega_t$ 
uniformly on $K_0$ and uniformly in $t\in [0,1]$.
\item $\theta_t$ has no zeros on $M$ for any $t\in [0,1]$
and its polar locus is $P_t=\{p_j(t)\}_j$.
\end{enumerate}
Assuming that the approximation of $f$ by $F$ is close enough,
every $\theta_t$ is full, and the implicit function theorem provides 
in view of conditions (II), (III), and \eqref{eq:omegatdef} 
a path $\beta:[0,1]\to\C^N$ such that $\beta(0)=0$, $g_t:=\Xi(t,\beta(t),\cdot)$ 
is uniformly close to $1$ on $K_0$ for all $t\in [0,1]$, and the continuous 
family of abelian differential $g_t\theta_t$ $(t\in [0,1])$ satisfies
\[
	\Re\int_{C_l} g_t\theta_t=\Re(\Pcal^t_l(1))
	=\Re\int_{C_l} \omega_t
	=\Re\int_{C_l} f_0\theta=0,\quad l=1,\ldots,\ell.
\]
(Cf.\ \cite[proof of Proposition 3.1]{AlarconLopez2025AMPA}.
Note that fullness of $\omega_t$ $(t\in[0,1])$ has been used here in an
important way to kill the real periods, which is seen in condition (II) above.)
Thus, taking into account conditions (I) and (III) and replacing $\theta_t$ by 
$g_t\theta_t$, we may assume in addition to conditions (i)--(iv) above that
\begin{equation}\label{eq:periodsthetat}
	\Re\int_{C_l}\theta_t=0 \quad 
	\text{for every $l\in \{1,\ldots,\ell\}$ and $t\in [0,1]$}.
\end{equation}
(Note that $g_0\equiv 1$.)
Proposition \ref{pr:multipliers} then furnishes a path of holomorphic 
functions $h_t:M\to\C^*$ $(t\in [0,1])$ such that 
$h_0\equiv 1$ and the following conditions hold for every $t\in [0,1]$. 
\begin{enumerate}[\rm (a)]
\item $\Re \int_{C_l} h_t\theta_t=0$ for every $l\in I$; 
take into account \eqref{eq:periodsthetat} and that 
$\Re \int_{C_l} \theta_0=2\Re\int_{C_l}\partial u_0=0$ for all $l\in I$.
\item $h_t-1$ vanishes to order $n_j$ at the point $p_j(t)$ for every 
$j=1,2,\ldots$.
\item $h_t$ is uniformly close to $1$ on $K_0$.
\end{enumerate}
Condition (b) implies that the full abelian differential $h_t\theta_t$ has the 
same residue as $\theta_t$ at the point $p_j(t)$ for every $j$ and $t$. 
By the construction, this agrees with the residue of $\theta_0=2\di u_0$ at 
$p_j=p_j(0)$, so its real part vanishes. Hence, condition (a) implies that 
$\Re(h_t\theta_t)$ is exact on $M$. 
Therefore, taking also (c) into account and replacing $\theta_t$ 
by $h_t\theta_t$, we may assume in addition to conditions (i)--(iv) above that 
$\Re(\theta_t)$ is exact on $M$ for every $t\in [0,1]$. It follows that the path 
of abelian differentials $\theta_t$ integrates by the Weierstrass
formula \eqref{eq:W-formula} to a path of maps $u_t \in \CMI_\s^\f(M,\R^n)$ 
with $2\di u_t=\theta_t$ for every $t\in [0,1]$, so $u_0=u$ (see (i)), 
$P(u_t)=P_t$ (see (iv)), and $u_t$ approximates $u_0$ on $K_0$
(see (iii) and \eqref{eq:omegatdef}) for every $t\in [0,1]$. 
In particular, the minimal surface $u_1\in \CMI_\s^\f(M,\R^n)$ 
has no singularities on $K_1$. This completes the first step of the construction. 

In the second step, we apply the same procedure, starting with $u_1$ 
and finding a path $\{u_t\}_{t\in [1,2]} \in \CMI_\s^\f(M,\R^n)$ such that 
for all $t\in [1,2]$ we have $P(u_t) \subset M\setminus K_1$,
$u_t$ approximates $u_1$ on $K_1$, and $u_2$ is nonsingular on $K_2$. 
Clearly, the induction may be continued so that the limit 
$u_\infty=\lim_{t\to\infty} u_t$ exists 
uniformly on compacts in $M$ and $u_\infty\in \CMI(M,\R^n)$.
Since all maps $u_t\in \CMI_\s^\f(M,\R^n)$ in the construction approximate 
$u_0$ on $K_0$, which is full, we can ensure that they all are full as well.

This completes the proof of Theorem \ref{th:sing}.

The following extension of this result, concerning moving complete ends 
of finite total curvature within the surface, follows by a straightforward 
modification of the proof. Recall the notation in \eqref{eq:Ecal}.

\begin{corollary}\label{co:sing}
If $M$, $n$, and $u$ are as in Theorem \ref{th:sing} and $C\subset M$ is a 
(possibly empty) closed discrete subset that is in bijection with a subset of 
$\Ecal_u^\infty$, then there is an isotopy $u_t \in \CMI_\s^\f(M,\R^n)$, 
$t\in [0,1]$, such that $u_0=u$ and $\Ecal_{u_1}^\infty=C$. 
\end{corollary}


\subsection*{Acknowledgements}
Alarc\'on is partially supported by the State Research Agency (AEI) via the 
grant no.\ PID2023-150727NB-I00, and the ``Maria de Maeztu'' Unit 
of Excellence IMAG, reference CEX2020-001105-M, 
funded by MICIU/AEI/10.13039/501100011033 and ERDF/EU, Spain.
Forstneri\v c is supported by the European Union 
(ERC Advanced grant HPDR, 101053085) 
and grants P1-0291 and N1-0237 from ARIS, Republic of Slovenia. 

%
%

\end{document}